\documentclass[11pt,reqno,oneside]{amsart}
\usepackage[totalwidth=14cm,totalheight=20cm, hmarginratio=1:1]{geometry}

\usepackage[english]{babel} 
\usepackage{esint}
\usepackage[T1]{fontenc} 
\usepackage{amsmath}
\usepackage{amssymb} 
\usepackage{float}
\usepackage{amsfonts}
\usepackage{mathtools}
\usepackage{bbm}
\usepackage[utf8]{inputenc}
\usepackage[mathcal]{eucal} 
\usepackage{enumerate}
\usepackage{amsthm} 
\usepackage[hidelinks]{hyperref}
\usepackage{xcolor}

\numberwithin{equation}{section}

\newtheorem{theorem}{Theorem}[section]
\newtheorem{lemma}[theorem]{Lemma}
\newtheorem{proposition}[theorem]{Proposition}
\newtheorem{corollary}[theorem]{Corollary}
\theoremstyle{remark}
\newtheorem{remark}[theorem]{Remark}

\newcommand{\R}{\mathbb{R}}

\newcommand{\SL}{\mathrm{SL}}

\newcommand{\Id}{\mathrm{Id}}

\newcommand{\Ree}{\mathcal{R}e}

\newcommand{\Area}{\mathrm{Area}}

\newcommand{\Teich}{\mathcal T(\Sigma)}
\newcommand{\Qthree}{\mathcal Q^3(\Teich)}

\newcommand{\dd}{\mathrm d}
\newcommand{\tr}{\operatorname{tr}}
\newcommand{\vol}{\operatorname{Vol}}
\newcommand{\Lie}{\mathcal L}
\newcommand{\ip}[2]{\left\langle #1,#2\right\rangle}
\newcommand{\norm}[1]{\left\lVert #1\right\rVert}

\DeclareMathAlphabet{\mathpzc}{OT1}{pzc}{m}{it}

\title[Comparison of pseudo-K\"ahler structures on the Hitchin component]{Global comparison of pseudo-K\"ahler structures on the $\SL(3,\R)$-Hitchin component}

\author{Andrea Tamburelli}
\address{Department of Mathematics,  Universit\`a di Pisa, Pisa, Italy}
\email{andrea\_tamburelli@libero.it}

\begin{document}

\begin{abstract}
We compare the closed $2$-form
$\omega_f$ constructed by Rungi--Tamburelli and Goldman's symplectic form
$\omega_G$ on the $\mathrm{SL}(3,\mathbb R)$-Hitchin component.  In particular, we establish that the semi-pseudo-K\"ahler structure on the
$\SL(3,\R)$-Hitchin component defined by Rungi--Tamburelli is non-degenerate
everywhere and, after aligning the normalization of the Killing form on $\mathfrak{sl}(3,\R)$, coincides with the one recently found by
Collier--Toulisse--Wentworth.
\end{abstract}
\maketitle

\tableofcontents

\section{Introduction}

Recent work of Collier--Toulisse--Wentworth \cite{CTW} constructs a joint moduli
space of Higgs bundles over Teichm\"uller space and produces pseudo-K\"ahler
metrics on several components of character varieties, including rank-two
higher Teichm\"uller spaces; the relevant two-form is induced by the
Atiyah--Bott--Goldman form.  It is worth mentioning that a similar construction was carried out independently by Alvarez-Cónsul--Garcia-Fernandez--García-Prada--Trautwein \cite{CGFGPT} and, for Hitchin components in rank two, the existence of pseudo-K\"ahler structures compatible with Goldman's symplectic form was also proved by El Emam--Sagman \cite{ES}.
Since the Atiyah--Bott--Goldman form depends on the choice of an invariant
bilinear form on the Lie algebra, all comparisons in this note use the
normalization induced by the trace pairing $(X,Y)\longmapsto \tr(XY)$.
Thus, when comparing with other constructions, their forms are understood to
be rescaled, if necessary, to this normalization.

\medskip
\noindent Earlier, Rungi--Tamburelli constructed a closed real $2$-form $\omega_f$ on the $\SL(3,\R)$-Hitchin component by an infinite-dimensional symplectic reduction, extending the para-hyperK\"ahler construction of Mazzoli--Seppi--Tamburelli for the deformation space of maximal globally hyperbolic anti-de Sitter three-manifolds \cite{MST}.  They proved that $\omega_f$ is compatible with the Labourie--Loftin complex structure $I$ and is non-degenerate near the Fuchsian locus \cite{RT}.  It remained open whether $\omega_f$ is non-degenerate on the whole Hitchin component and how it is related to $\omega_G$.  Collier--Toulisse--Wentworth also noted that the relation between the two constructions was unclear.

\medskip
\noindent The purpose of this note is to prove that the Rungi--Tamburelli form
is exactly Goldman's form.  The proof has
two parts.  First, we compute the contraction of both forms with the
infinitesimal generator of the circle action in the fiber of the bundle of holomorphic cubic differentials and prove that their Hamiltonians agree.  Second, we prove a general comparison principle: two closed real $2$-forms compatible with the complex structure $I$ on the total space
of a holomorphic vector bundle which have the same contraction with the
infinitesimal generator of the circle action differ by the pullback of a form on the zero section.  Since the two forms agree on the Fuchsian locus, they agree everywhere.

\subsection{Statement and conventions}

Let $\Sigma$ be a closed oriented surface of genus $g\geq 2$.  We identify the Hitchin component with the total space of the bundle of holomorphic cubic differentials $\Qthree$ over the Teichm\"uller space of $\Sigma$ via the Labourie-Loftin parameterization. The complex structure on this holomorphic vector bundle will be denoted by
$I$. 

\noindent We fix an area form $\rho$ satisfying
\begin{equation}\label{eq:area-normalization}
    \vol(\Sigma,\rho)=-2\pi\chi(\Sigma),
\end{equation}
so that the constant $c=2\pi\chi(\Sigma)/\vol(\Sigma,\rho)$ appearing in the
Rungi--Tamburelli construction is $c=-1$.
For a complex structure $J$, put
\[
    g=g_J:=\rho(\,\cdot\,,J\,\cdot\,).
\]

\noindent We use $q$ for the cubic differential as in Rungi--Tamburelli.  Note that this implies that the Pick differential of the hyperbolic affine sphere is $q_{Pick}=\frac{1}{\sqrt{2}}q$. This can be easily understood from the fact that in the Rungi--Tamburelli conventions the Wang equation takes the form
\[
    \kappa_{h}-\|q\|^{2}_{h}=-1,
\]
where $h$ is the Blaschke metric, whereas $q_{Pick}$ satisfies
\[
    \kappa_{h}-2\|q_{Pick}\|^{2}_{h}=-1 .
\]
Set $A=g^{-1}\Ree(q)$ and
\begin{equation}\label{eq:t-definition}
    t:=\norm{A}_0^2=\frac18\norm{A}_g^2
      =\frac12\norm{q}_g^2.
\end{equation}
The Blaschke metric can be written as
\begin{equation}\label{eq:h-conformal}
    h=e^{F(t)}g,
\end{equation}
where $F$ is determined by
\begin{equation}\label{eq:F-equation}
    -e^{-F(t)}-2t e^{-3F(t)}+1=0,
    \qquad F(0)=0.
\end{equation}

The semi-pseudo-K\"ahler form $\omega_f$ defined by Rungi--Tamburelli is the closed real $2$-form obtained by integrating over $\Sigma$ the following pointwise expression (using the area form $\rho$)
\begin{align}\label{eq:omega-f-formula}
\omega_f\big((\dot J,\dot A),(\dot J',\dot A')\big)
={}&(f(t)-1)\ip{\dot J}{J\dot J'}
 -\frac{f'(t)}{12}\ip{\dot A_{\mathrm{tr}}}{*_J\dot A'_{\mathrm{tr}}}
 -\frac{f'(t)}6\ip{\dot A_0}{\dot A'_0J}.
\end{align}
To clarify the notation, let $(J_s,q_s)$ be a smooth path through
$(J,q)$ and put
\[
    g_s=g_{J_s},\qquad C_s=\Ree(q_s),\qquad A_s=g_s^{-1}C_s.
\]
We write
\[
    \dot C=\left.\frac{\dd}{\dd s}\right|_{s=0}C_s,
    \qquad
    \dot A:=g^{-1}\dot C=g^{-1}\Ree(\dot q).
\]
Thus $\dot A$ denotes the variation of the cubic tensor with the index raised
by the fixed metric $g=g_J$; it is not, in general, the derivative of the path
$A_s$.  The latter will be computed explicitly in Section~\ref{sec:Goldman-Hamiltonian}.
Moreover, the function $f$ appearing in the formula is defined by
\begin{equation}\label{eq:integral_f}
f(t)
=
-t^{1/3}\int_0^t s^{-1/3}F'(s)\,ds,
\qquad f(0)=0.
\end{equation}

\noindent The main result of this note is the following, which answers Question~1 of \cite{RT} and resolves the comparison issue noted in \cite{CTW}.

\begin{theorem}\label{thm:main} The $2$-form $\omega_f$ coincides
with the Goldman symplectic form at every point.  Consequently, $\omega_f$ is
non-degenerate everywhere.  After possibly normalizing the Killing form used to define the Atiyah-Bott-Goldman symplectic form, the associated pseudo-K\"ahler structure coincides with the one
studied by Collier--Toulisse--Wentworth on the $\SL(3,\R)$-Hitchin component.
\end{theorem}

The key idea towards the proof of Theorem \ref{thm:main} is showing that the natural circle action given by rotation along the fibers of $\Qthree$ is Hamiltonian for both $\omega_{f}$ and $\omega_{G}$ and they have the same Hamiltonian. The result will then follow by comparing the two forms at the Fuchsian locus and by a general identity principle for pseudo-K\"ahler structures.

\subsection*{Acknowledgment} The author thanks Diego Conti for suggesting that a uniqueness statement for pseudo-K\"ahler symplectic forms may hold, which led to Lemma \ref{lem:radial}. He is also grateful to Nicholas Rungi for his comments on the first draft.

\section{The Hamiltonian circle action for $\omega_{f}$}

We start by recalling the Hamiltonian circle action on $\omega_{f}$. This was already studied in \cite[Theorem C]{RT} (see also \cite[Theorem B]{RTtorus}), but we will need to write the Hamiltonian in a slightly different way using new relations between the functions $f$ and $F$ defined in the introduction.

\begin{lemma}\label{lem:f-explicit}
For $c=-1$, one has
\begin{equation}\label{eq:f-explicit}
    f(t)=-\frac32\big(e^{F(t)}-1\big)
        =-3t e^{-2F(t)}.
\end{equation}
In particular,
\begin{equation}\label{eq:fprime-zero}
    f'(0)=-3.
\end{equation}
\end{lemma}

\begin{proof} By the Implicit Function Theorem applied to Equation \eqref{eq:F-equation}, the function $F(t)$ is smooth, $F(0)=0$ and $F'(0)=2$. 
Multiplying \eqref{eq:F-equation} by $e^{3F}$ gives
\begin{equation}\label{eq:F-algebraic}
    2t=e^{2F}(e^F-1),
\end{equation}
so
\begin{equation}\label{eq:eF-minus-one}
    e^F-1=2t e^{-2F}.
\end{equation}
Define $\widetilde f(t)=-\frac32(e^{F(t)}-1)$.  It satisfies
$\widetilde f(0)=0$.  Differentiating \eqref{eq:F-algebraic} gives
\[
    2=e^{2F}(3e^F-2)F'.
\]
Using this identity and \eqref{eq:F-algebraic}, a direct simplification gives
\[
    \widetilde f'
    =-F'+\frac{\widetilde f}{3t}.
\]
On the other hand, differentiating \eqref{eq:integral_f} we see that $f$ satisfies the same first-order differential equation. Moreover, smoothness of $F$ implies that $f$ admits a finite right derivative at $t=0$. Indeed, setting $s=tu$ in Equation \eqref{eq:integral_f}, we have
\begin{align*}
    \lim_{t\to 0^{+}}\frac{f(t)}{t}&=\lim_{t\to 0^{+}}-t^{-2/3}\int_0^t s^{-1/3}F'(s)\,ds \\
    &=\lim_{t\to 0^{+}}-\int_0^1 u^{-1/3}F'(tu)\,ds \\
    &=-F'(0)\int_0^1 u^{-1/3} du = -3 . 
\end{align*}
As a consequence, the function \(u=f-\widetilde f\) admits right derivative at $t=0$ as well and, for $t>0$, is a solution of the first-order differential equation
\[
u'(t)=\frac{u(t)}{3t}.
\]
Since all solutions for $t>0$ are of the form \(u(t)=Ct^{1/3}\) for some constant \(C \in \R\), the only one admitting right derivative at $t=0$ is $u(t)=0$, hence $f(t)=\widetilde{f}(t)$ and Equation \eqref{eq:f-explicit} follows.
\end{proof}

\noindent The circle action on $\Qthree$ is
\begin{equation}\label{eq:circle-action}
    \Psi_\theta(J,q)=(J,e^{-i\theta}q),
\end{equation}
and its infinitesimal generator is
\begin{equation}\label{eq:X-generator}
    X_{(J,A)}=(0,-AJ).
\end{equation}
By \cite[Theorem C]{RT}, its Hamiltonian for $\omega_f$ is
\[
    H_f(J,A)=\frac23\int_\Sigma f(t)\,\rho.
\]
\noindent Define the energy functional
\begin{equation}\label{eq:E-definition}
    \mathcal E(J,A):=\int_\Sigma t e^{-2F(t)}\,\rho.
\end{equation}
Lemma~\ref{lem:f-explicit} immediately yields:

\begin{corollary}\label{cor:Hf}
The Hamiltonian of the circle action for $\omega_f$ is
\begin{equation}\label{eq:Hf-energy}
    H_f=-2\mathcal E,
    \qquad
    \iota_X\omega_f=-2\,\dd\mathcal E.
\end{equation}
\end{corollary}

\begin{remark} Lemma~\ref{lem:f-explicit} also clarifies the geometric meaning of the Hamiltonian $H_{f}$. Indeed,
\[
    H_{f}(J,A)=\frac{2}{3}\int_{\Sigma}f(t) \rho =- \int_{\Sigma}(e^{F(t)}-1)\rho =
    -\Area(\Sigma,h)-2\pi\chi(\Sigma),
\]
hence the Hamiltonian can be chosen to be the area of the Blaschke metric. In fact, a Hamiltonian is always defined up to an additive constant and the unpleasant negative sign in the above formula is due to our convention of the circle action being $(J,q) \mapsto (J,e^{-i\theta}q)$, which we committed to in \cite{RTtorus}. This result should be compared to the case of the deformation space of anti-de Sitter structures (\cite{MST}) where the role of the Blaschke metric is played by the induced metric on the maximal surface.
\end{remark}

\section{The corrected comparison on the Fuchsian locus}

Let $j:\Teich\hookrightarrow\Qthree$ be the zero section.  At a point
$(J,0)$, the tangent space decomposes into horizontal and vertical directions.
The tangent vectors satisfy (\cite{RT})
\[
    \operatorname{div}_g\dot J=0,
    \qquad
    \dd^\nabla\dot A_0=0,
\]
and the trace part of $\dot A$ vanishes.\\

We now compare $\omega_f$ with $\omega_G$ on the Fuchsian locus, fixing in the meantime a harmless typo in \cite{RT}. 

\begin{proposition}\label{prop:fuchsian-comparison}
The forms $\omega_f$ and $\omega_G$ agree along the whole tangent space over the Fuchsian locus.
\end{proposition}

\begin{proof} In \cite{RT} there is a typo in the scaling factor: the correct relation as mentioned before is $q_{Pick}=\frac{1}{\sqrt{2}}q$, hence the Blaschke connection is $\nabla=\nabla^{h}+h^{-1}\Ree(q_{Pick})=\nabla^{h}+h^{-1}\frac{1}{\sqrt{2}}\Ree(q)$. In fact, only with this correct scaling factor the two forms agree along vertical directions. We go over the proof again for the convenience of the reader, but the computation in \cite{RT} is essentially correct. \\

\noindent Goldman's symplectic form admits a description directly in terms of
projectively flat torsion-free affine connections. Let \(\nabla\) be
the Blaschke connection on \(T\Sigma\), and let
\(\alpha,\beta\in\Omega^1(\Sigma,\mathrm{End}(T\Sigma))\) be infinitesimal
variations of \(\nabla\). Then Goldman's symplectic form can be computed as \cite{Goldman}
\begin{equation}\label{eq:Goldman}
\omega_G(\alpha,\beta)
=
\int_\Sigma
\left(
\tr(\alpha\wedge\beta)
-\frac13\tr(\alpha)\wedge\tr(\beta)
\right).
\end{equation}
Under the holonomy correspondence, this agrees with the Goldman form
on the \(\SL(3,\mathbb R)\)-character variety associated with the
invariant pairing \((X,Y)\mapsto\tr(XY)\).
\noindent We apply this formula to variations of the Blaschke connection distinguishing horizontal and vertical variations. \\

\noindent\emph{Vertical--vertical block.}
For vertical vectors $(0,\dot A_0)$ and $(0,\dot A'_0)$, the corresponding
variations of the Blaschke connection are
\[
    \dot\nabla=\frac1{\sqrt2}\dot A_0,
    \qquad
    \dot\nabla'=\frac1{\sqrt2}\dot A'_0.
\]
Goldman's gauge-theoretic formula therefore gives
\begin{align*}
    \omega_G(\dot\nabla,\dot\nabla')
    =\frac12\int_\Sigma\tr(\dot A_0\wedge\dot A'_0)
    =\frac12\int_\Sigma\ip{\dot A_0}{\dot A'_0J}\,\rho.
\end{align*}
On the other hand, \eqref{eq:omega-f-formula} and $f'(0)=-3$ give
\[
    \omega_f\big((0,\dot A_0),(0,\dot A'_0)\big)
    =-\frac{f'(0)}6
      \int_\Sigma\ip{\dot A_0}{\dot A'_0J}\,\rho,
\]
which is the same expression.

\medskip
\noindent\emph{Horizontal--vertical block.}
A horizontal vector $(\dot J,0)$ induces
\[
    \dot\nabla=-\frac12J(\nabla\dot J),
\]
whereas a vertical vector induces $\dot\nabla'=\dot A'_0/\sqrt2$.  Hence
\begin{align*}
\omega_G(\dot\nabla,\dot\nabla')
&=-\frac1{2\sqrt2}
  \int_\Sigma\tr\big(J\nabla\dot J\wedge\dot A'_0\big)\\
&=-\frac1{2\sqrt2}
  \int_\Sigma\tr\big(\nabla(J\dot J)\wedge\dot A'_0\big)\\
&=\frac1{2\sqrt2}
  \int_\Sigma\tr\big(J\dot J\wedge\dd^\nabla\dot A'_0\big)=0.
\end{align*}
The corresponding mixed block of $\omega_f$ vanishes directly from
\eqref{eq:omega-f-formula} at $A=0$.

\medskip
\noindent\emph{Horizontal--horizontal block.}
This part of the calculation is unaffected by the fibre normalization.  For
horizontal vectors $(\dot J,0)$ and $(\dot J',0)$,
\[
    \dot\nabla=-\frac12J(\nabla\dot J),
    \qquad
    \dot\nabla'=-\frac12J(\nabla\dot J').
\]
At the Fuchsian locus, $q=0$ and $F(0)=0$, so the Blaschke metric satisfies
$h=g$.  Wang's equation therefore gives $K_g=-1$, and hence
$R^\nabla=-J\otimes\rho$.  Integrating by parts exactly as in the
proof of Theorem D of \cite{RT},
\begin{align*}
\omega_G(\dot\nabla,\dot\nabla')
&=\frac14\int_\Sigma
  \tr\big(J\nabla\dot J\wedge J\nabla\dot J'\big)\\
&=-\frac14\int_\Sigma
  \tr\big(J\dot J\wedge[-R^\nabla,J\dot J']\big)\\
&=-\frac12\int_\Sigma\tr(J\dot J\dot J')\,\rho\\
&=\omega_f\big((\dot J,0),(\dot J',0)\big).
\end{align*}
The last equality is precisely the horizontal part of
\eqref{eq:omega-f-formula} at $A=0$ and $f(0)=0$.
\end{proof}

\section{The Hamiltonian circle action for $\omega_G$}
\label{sec:Goldman-Hamiltonian}

\noindent We now work at an arbitrary point $(J,A)$.  The Blaschke connection is
\[
     \nabla=\nabla^h+B, 
     \qquad
    B=\frac{e^{-F}}{\sqrt2}A
      =\frac{e^{-F}}{\sqrt2}g_J^{-1}\Ree(q)
\]
We first record carefully how $B$ varies.  Along a path $(J_s,q_s)$ as in the
introduction, differentiating
\[
    g_s(U,V)=\rho(U,J_sV)
\]
gives
\[
    \dot g(U,V)=\rho(U,\dot J V)=g(U,-J\dot J V),
    \qquad
    g^{-1}\dot g=-J\dot J.
\]
Therefore
\begin{equation}\label{eq:actual-A-variation}
    \left.\frac{\dd}{\dd s}\right|_{s=0}A_s
    =-g^{-1}\dot g\,A+g^{-1}\dot C
    =\dot A+J\dot J A.
\end{equation}

\medskip
\noindent Fibre rotation fixes $J$, $t$, $F(t)$, and $h$.  Consequently, the
variation along an orbit of the circle action generated by $X$ is
\begin{equation}\label{eq:B-X-variation}
    \dot B=-BJ,
    \qquad
    \dot\nabla=-BJ.
\end{equation}
For an arbitrary tangent vector $Y=(\dot J,\dot A)$, equations
\eqref{eq:actual-A-variation} and $B=e^{-F}A/\sqrt2$ give
\begin{equation}\label{eq:B-general-variation}
    \dot\nabla=\dot\nabla^h+\dot B,
    \qquad
    \dot B=\frac{e^{-F}}{\sqrt2}
    \bigl(\dot A+J\dot J A-\dot F A\bigr),
    \qquad
    \dot F=F'(t)\dot t.
\end{equation}
Because $B(U)$ is $h$-symmetric while $J$ is $h$-skew-symmetric,
\[
    \tr(B(U)J)=0.
\]
Thus the trace-correction term in Goldman's formula vanishes when one
argument is $X$, and
\begin{equation}\label{eq:Goldman-contraction-split}
\omega_G(X,Y)
=-\int_\Sigma\tr(BJ\wedge\dot\nabla^h)
 -\int_\Sigma\tr(BJ\wedge\dot B).
\end{equation}

\begin{lemma}\label{lem:LC-vanishing}
For every tangent vector $Y$,
\begin{equation}\label{eq:LC-vanishing}
    \int_\Sigma\tr(BJ\wedge\dot\nabla^h)=0.
\end{equation}
\end{lemma}

\begin{proof}
Write
\[
    D:=\dot\nabla^h\in
    \Omega^1\bigl(\Sigma,\operatorname{End}(T\Sigma)\bigr)
\]
and define the cubic tensor
\[
    \mathcal{C}(U,V,W):=h(B(U)V,W).
\]
This is the real part of the Pick differential.  In particular,
$\mathcal C$ is totally symmetric and trace-free, and it satisfies the
Codazzi equation
\begin{equation}\label{eq:C-Codazzi}
    \nabla_\ell\mathcal C_{ijk}
    =\nabla_i\mathcal C_{\ell jk}.
\end{equation}
All covariant derivatives in this proof are taken with respect to
$\nabla^h$.

\smallskip
\noindent\emph{Step 1: expansion of the wedge product.}
Fix $p\in\Sigma$ and choose a positively oriented $h$-orthonormal frame
$e_2=Je_1$ near $p$.  For endomorphism-valued one-forms $\alpha$ and $\beta$,
we use the convention
\[
\tr(\alpha\wedge\beta)(e_1,e_2)
=\tr\bigl(\alpha(e_1)\beta(e_2)
          -\alpha(e_2)\beta(e_1)\bigr).
\]
The cubic-differential identities give
\begin{equation}\label{eq:B-complex-linear}
    B(JU)=B(U)J.
\end{equation}
Consequently,
\[
    B(e_1)J=B(e_2),
    \qquad
    B(e_2)J=-B(e_1),
\]
and hence
\begin{align*}
\tr(BJ\wedge D)(e_1,e_2)
&=\tr\bigl(B(e_1)J D_{e_2}-B(e_2)J D_{e_1}\bigr)\\
&=\tr\bigl(B(e_2)D_{e_2}+B(e_1)D_{e_1}\bigr).
\end{align*}
Thus
\begin{equation}\label{eq:wedge-full-contraction}
    \tr(BJ\wedge D)
    =\sum_{i=1}^2\tr\bigl(B(e_i)D_{e_i}\bigr)\,\dd A_h.
\end{equation}
Because each $B(e_i)$ is $h$-symmetric,
\begin{align*}
\tr\bigl(B(e_i)D_{e_i}\bigr)
&=\sum_j h\bigl(B(e_i)D_{e_i}e_j,e_j\bigr)\\
&=\sum_j h\bigl(D_{e_i}e_j,B(e_i)e_j\bigr).
\end{align*}
Put
\[
    D_{ijk}:=h(D_{e_i}e_j,e_k),
    \qquad
    \mathcal C_{ijk}:=h(B(e_i)e_j,e_k).
\]
Equation \eqref{eq:wedge-full-contraction} becomes
\begin{equation}\label{eq:wedge-C-D}
    \tr(BJ\wedge D)
    =\mathcal C^{ijk}D_{ijk}\,\dd A_h.
\end{equation}

\smallskip
\noindent\emph{Step 2: contraction of the Levi-Civita variation formula.}
The standard first-variation formula for the Levi-Civita
connection is
\begin{equation}\label{eq:LC-variation}
2h(D_UV,W)
=(\nabla_U \dot{h})(V,W)+(\nabla_V \dot{h})(U,W)-(\nabla_W \dot{h})(U,V).
\end{equation}
In indices,
\begin{equation}\label{eq:LC-variation-indices}
    2D_{ijk}
    =\nabla_i \dot{h}_{jk}+\nabla_j \dot{h}_{ik}-\nabla_k \dot{h}_{ij}.
\end{equation}
Contracting with the completely symmetric tensor $\mathcal C^{ijk}$ gives
\begin{align*}
2\mathcal C^{ijk}D_{ijk}
={}&\mathcal C^{ijk}\nabla_i k_{jk}
  +\mathcal C^{ijk}\nabla_j k_{ik}
  -\mathcal C^{ijk}\nabla_k k_{ij}.
\end{align*}
The three contractions on the right are equal: in the second one exchange
$i$ and $j$, and in the third make the cyclic relabelling
$(i,j,k)\mapsto(k,i,j)$.  It follows that
\[
    \mathcal C^{ijk}D_{ijk}
    =\frac12\mathcal C^{ijk}\nabla_i k_{jk}.
\]
Together with \eqref{eq:wedge-C-D}, this proves the pointwise identity
\begin{equation}\label{eq:LC-C-contraction}
    \tr(BJ\wedge\dot\nabla^h)
    =\frac12\mathcal C^{ijk}\nabla_i k_{jk}\,\dd A_h.
\end{equation}

\smallskip
\noindent\emph{Step 3: a trace-free Codazzi cubic is divergence-free.}
We claim that
\begin{equation}\label{eq:C-divergence-free}
    \nabla^i\mathcal C_{ijk}=0.
\end{equation}
Indeed, using first the symmetry of $\mathcal C$ and then the Codazzi identity
\eqref{eq:C-Codazzi},
\begin{align*}
\nabla^i\mathcal C_{ijk}
&=h^{i\ell}\nabla_\ell\mathcal C_{ijk}
 =h^{i\ell}\nabla_\ell\mathcal C_{jik}\\
&=h^{i\ell}\nabla_j\mathcal C_{\ell ik}
 =\nabla_j\bigl(h^{i\ell}\mathcal C_{\ell ik}\bigr).
\end{align*}
The last expression vanishes because $\mathcal C$ is trace-free.  This proves
\eqref{eq:C-divergence-free}.

\smallskip
\noindent\emph{Step 4: integration by parts.}
Since $\Sigma$ is closed, Equation \eqref{eq:LC-C-contraction} gives
\begin{align*}
\int_\Sigma\tr(BJ\wedge\dot\nabla^h)
&=\frac12\int_\Sigma
  \mathcal C^{ijk}\nabla_i k_{jk}\,\dd A_h\\
&=-\frac12\int_\Sigma
  (\nabla_i\mathcal C^{ijk})k_{jk}\,\dd A_h=0,
\end{align*}
where the final equality follows from
\eqref{eq:C-divergence-free}.  This proves the lemma.
\end{proof}

\begin{lemma}\label{lem:pointwise-identities}
For every tangent vector $(\dot J,\dot A)$, the following pointwise identities
hold:
\begin{align}
    \tr(AJ\wedge\dot A)
    &=\ip{A}{\dot A_0}\,\rho,
    \label{eq:trace-A-dotA}\\
    \tr(AJ\wedge A)
    &=\norm{A}_g^2\,\rho=8t\,\rho,
    \label{eq:trace-A-A}\\
    \tr\bigl(AJ\wedge J\dot JA\bigr)
    &=0,
    \label{eq:trace-A-JdotJA}\\
    \dot t&=\frac14\ip{A}{\dot A_0}.
    \label{eq:t-variation}
\end{align}
\end{lemma}

\begin{proof}
Fix a point and choose a positively oriented $g$-orthonormal basis
$e_2=Je_1$.  Put
\[
    A_i:=A(e_i),\qquad \dot A_i:=\dot A(e_i).
\]
The cubic-differential identity $A(JU)=A(U)J$ gives
\[
    A_2=A_1J,\qquad A_1J=A_2,\qquad A_2J=-A_1.
\]
Therefore
\begin{align*}
\tr(AJ\wedge\dot A)(e_1,e_2)
&=\tr(A_1J\dot A_2-A_2J\dot A_1)\\
&=\tr(A_1\dot A_1+A_2\dot A_2)
 =\ip{A}{\dot A}.
\end{align*}
The trace part of $\dot A(X)$ is a scalar multiple of the identity, whereas
each $A(X)$ is trace-free.  Hence $A$ is orthogonal to $\dot A_{\rm tr}$,
which proves \eqref{eq:trace-A-dotA}.  Taking $\dot A=A$ proves
\eqref{eq:trace-A-A}, and the equality with $8t\rho$ follows from the
definition of $t$.

\noindent For \eqref{eq:trace-A-JdotJA}, set $S:=J\dot J$.  Expanding as above and using
cyclicity of the trace gives
\begin{align*}
\tr\bigl(AJ\wedge SA \bigr)(e_1,e_2)
&=\tr(A_2SA_2)+\tr(A_1SA_1)\\
&=\tr\bigl(S(A_1^2+A_2^2)\bigr).
\end{align*}
A trace-free symmetric endomorphism of a two-dimensional Euclidean vector
space has square equal to a scalar multiple of the identity.  Thus
$A_1^2+A_2^2$ is a scalar multiple of the identity.  On the other hand,
differentiating $J^2=-\Id$ and taking the trace gives
$\tr(J\dot J)=0$.  Hence the last expression vanishes.

\noindent Equation \eqref{eq:t-variation} was proved in \cite[Lemma 3.22]{RTtorus}.
\end{proof}

\noindent Thus, using \eqref{eq:B-general-variation}, we find
\begin{align}
\tr(BJ\wedge\dot B)
&=\frac12e^{-2F}
  \left(\tr(AJ\wedge\dot A)-\dot F\tr(AJ\wedge A)\right)
\notag\\
&=\frac12e^{-2F}(1-2tF'(t))
  \ip{A}{\dot A_0}\,\rho.
\label{eq:algebraic-term}
\end{align}
On the other hand, differentiating \eqref{eq:E-definition} gives
\begin{equation}\label{eq:dE}
\dd\mathcal E(Y)
=\frac14\int_\Sigma e^{-2F}(1-2tF'(t))
  \ip{A}{\dot A_0}\,\rho.
\end{equation}
Combining \eqref{eq:Goldman-contraction-split},
Lemma~\ref{lem:LC-vanishing}, \eqref{eq:algebraic-term}, and
\eqref{eq:dE}, we obtain:

\begin{proposition}\label{prop:Goldman-Hamiltonian}
The fibre-rotation action is Hamiltonian for $\omega_G$, with Hamiltonian
\begin{equation}\label{eq:HG}
    H_G=-2\mathcal E.
\end{equation}
Equivalently,
\begin{equation}\label{eq:contractions-equal}
    \iota_X\omega_G=-2\,\dd\mathcal E
    =\iota_X\omega_f.
\end{equation}
\end{proposition}

\section{A uniqueness lemma}

We now use the global result of El Emam--Sagman \cite{ES}: Goldman's
symplectic form is compatible with the Labourie--Loftin complex structure.
The same is true of $\omega_f$ by construction.

\begin{lemma}\label{lem:radial}
Let $\pi:E\to M$ be a holomorphic vector bundle, let $I$ be its complex
structure, and let $X$ be the infinitesimal generator of the circle action on the fibers.  Suppose that $\Omega_0$
and $\Omega_1$ are closed real $2$-forms on $E$ compatible with $I$ such that
\[
    \iota_X\Omega_0=\iota_X\Omega_1.
\]
Then
\begin{equation}\label{eq:radial-pullback}
    \Omega_0-\Omega_1
    =\pi^*j^*(\Omega_0-\Omega_1),
\end{equation}
where $j:M\hookrightarrow E$ is the zero section.  In particular, if their
restrictions to the zero section agree, then $\Omega_0=\Omega_1$ everywhere.
\end{lemma}

\begin{proof}
Set $\Delta=\Omega_0-\Omega_1$.  Then
\[
    \dd\Delta=0,
    \qquad
    \iota_X\Delta=0.
\]
Let $R$ be the infinitesimal generator of the $\R^*$-action $m_{s}:(J,A) \mapsto (J,sA)$.  Since $I(0,\dot{A})=(0,-\dot{A}J)$ (\cite[Equation 4.9]{RT}), we have
\begin{equation}\label{eq:R-IX}
    R=(0,A)=-I(0,-AJ)=-IX.
\end{equation}
Since $\Omega_{0}$ and $\Omega_{1}$ are both compatible with $I$, we have
\[
    \Delta(IU,V)=-\Delta(U,IV).
\]
Therefore, for every $Y$,
\[
    \Delta(R,Y)=-\Delta(IX,Y)=\Delta(X,IY)=0,
\]
so $\iota_R\Delta=0$.  Cartan's formula gives
\[
    \Lie_R\Delta=\dd(\iota_R\Delta)+\iota_R\dd\Delta=0.
\]
Hence, for all $s>0$,
\begin{equation}\label{eq:dilation-invariance}
    m_s^*\Delta=\Delta.
\end{equation}
Taking the limit as $s \to 0$, we see that $m_{0}=j \circ \pi$, thus 
\[
    \Delta=\lim_{s\to 0^+}m_{s}^{*}\Delta=(j\circ \pi)^{*}\Delta=\pi^{*}j^{*}\Delta
\]
as claimed.
\end{proof}

\section{Proof of the global comparison}

\begin{proof}[Proof of Theorem~\ref{thm:main}]
Set
\[
    \Delta:=\omega_f-\omega_G.
\]
Both forms are closed real forms compatible with $I$.  Proposition
\ref{prop:Goldman-Hamiltonian} gives $\iota_X\Delta=0.$
Proposition~\ref{prop:fuchsian-comparison} gives $j^*\Delta=0$.
Applying Lemma~\ref{lem:radial}, we conclude that $\Delta=0$ on the whole
bundle, hence $\omega_f=\omega_G$. Goldman's form is symplectic, so $\omega_f$ is non-degenerate everywhere. Moreover, after normalizing the Atiyah--Bott--Goldman form in the
Collier--Toulisse--Wentworth construction by the same trace pairing, the two
pseudo-K\"ahler structures have the same symplectic form and complex
structure, and therefore coincide.
\end{proof}

 \end{document}